\documentclass[12pt]{article}
\usepackage{amsmath,amssymb,amsbsy,amsfonts,amsthm,latexsym,amsopn,amstext,
                               amsxtra,euscript,amscd}
\begin{document}

\newtheorem{theorem}{Theorem}
\newtheorem{lemma}[theorem]{Lemma}
\newtheorem{claim}[theorem]{Claim}
\newtheorem{cor}[theorem]{Corollary}
\newtheorem{prop}[theorem]{Proposition}
\newtheorem{definition}{Definition}
\newtheorem{question}[theorem]{Open Question}
\newtheorem{rem}[theorem]{Remark}
\newtheorem{obs}[theorem]{Observation}
\newtheorem{example}[theorem]{Example}

\def\e{{\epsilon}}
\def\a{{\alpha}}
\def\l{{\lambda}}
\def\lf{{\lfloor}}
\def\rf{{\rfloor}}

\title{On a family of recurrences that includes the Fibonacci and the Narayana recurrences} 

\date{}

\author{
{\sc Christian ~Ballot} \\
{Universit\'e de Caen} \\
{D\'epartement de Math\'ematiques et Informatique}\\}
\maketitle
\begin{abstract}
We survey and prove properties a family of recurrences bears in relation to  
integer representations, compositions, Pascal's triangle,  
sums of digits, Nim games and Beatty sequences. 
\end{abstract}

\section{Introduction}
\label{sec:intro} The Fibonacci sequence $(F_k)_{k\ge0}$ defined by $F_0=0$, $F_1=1$ 
and $F_{k+2}=F_{k+1}+F_k$ is well known for its various natural and man-made occurrences 
and its many arithmetic properties. Perhaps the three most commonly visited cubic generalizations 
of the Fibonacci sequence are -- in this order -- the tribonacci, the Narayana and the Padovan sequences. 
They all have initial values $0$, $0$ and $1$ and their respective characteristic polynomials are 
$x^3-x^2-x-1$, $x^3-x^2-1$ and 
$x^3-x-1$. In fact, to each of these cubics we may naturally associate an infinite family  
of polynomials and recurrences: With each polynomial $f_q$ of degree $q\ge1$, we associate  
its fundamental recurrence, which has $q-1$ zeros followed by a one as initial values and $f_q$ 
as characteristic polynomial. Thus, we distinguish the three families

\medskip 

{\bf 1.} (the tribonacci family) with polynomials $x^q-\sum_{i=0}^{q-1}x^i$, ($q\ge2$). 

\medskip 

{\bf 2.} (the Narayana family) with $g_q(x):=x^q-x^{q-1}-1$, ($q\ge1$). 

\medskip

{\bf 3.} (the Padovan family) with polynomials $x^q-x-1$, ($q\ge2$).  

\medskip

 Besides this introduction, the paper contains two more sections. 
Section 2 presents six theorems that concern the Narayana family. Proofs are given for 
four theorems and references for the two remaining ones. There is little claim of novelty 
in those theorems. The interest of the paper lies more  
in the adopted perspective, in the combinatorial or counting nature of several proofs,  
in the way these theorems are brought and linked together, and in the further research 
questions this posture naturally raises. Bits of these 
theorems are sometime variously scattered, and we quote references we are aware of as we go 
along. However, Theorem \ref{thm:FarDif} and Theorem \ref{thm:Id+}, (3), can be found respectively 
in \cite{DeDo, BiSp}, two papers having to do with the same Narayana family of recurrences. 

\smallskip
 
 Given $q\ge1$, we will designate by the letter $G$, or the symbol $G^q$, the fundamental recurrence of $g_q(x)$. 
Thus, $G$ has characteristic polynomial $g_q$, 
i.e., $G_{k+q}=G_{k+q-1}+G_k$ for all $k\ge0$.  The first few $G$ values $G_0$, $G_1$, etc.,  
are $0,\dots,0,1,1,\dots,1,2,\dots$, with $q-1$ initial zeros followed by $q$ ones. 
Thus, $G_{q-1}=G_{2q-2}=1$. For $q=1,\,2$ and $3$, we find respectively $G^1_k=2^k$, 
$G^2_k=F_k$ and $G^3_k=N_k$, where $F_k$ and $N_k$ are the $k$th Fibonacci and Narayana numbers. 
\footnote{Many facts about these three sequences may be found in the OEIS \cite{Slo}; 
see also \cite{AJ} and our end-of-paper note for historical remarks}

 The theorems we present are better known to hold for the binary and Fibonacci cases.  

 The binary and the standard Fibonacci, or Zeckendorf \cite{Zec}, representations 
of integers extend to a unique representation into distinct $G$ numbers, for arbitrary $q\ge1$. 
We call here such a representation a $q$-{\it representation}. 
This is shown in Theorem \ref{thm:Zec+} using a counting argument. We present 
the so-called {\it far-difference} representation which uses signed $G$ numbers in Theorem \ref{thm:FarDif}. 
Theorem \ref{thm:Comp+} 
shows how $G$ numbers occur naturally in counting compositions of integers into sets  
of parts related to $q$. Theorem \ref{thm:Id+} proves a choice of three identities.  
One identity connects sums of binomial coefficients along lines of Pascal's triangle to 
$G$-numbers; another finds the total number of $G$-summands  over all 
$q$-representations of all integers up to $G_k$ in terms of a simple weighted 
sum of binomial coefficients. Theorem \ref{thm:Nim} proposes simple variants of the game 
of Fibonacci Nim when $1\le q\le3$ which bear a common proof and a similar winning 
strategy. Kimberling gave an expression involving Fibonacci numbers for all functions composed of several of 
the two Beatty sequences $a(n)=\lf n\a\rf$ and $b(n)=\lf n\a^2\rf$, where $\a=(1+\sqrt{5})/2$. 
This theorem generalizes to all pairs $a(n)=\lf n\a\rf$ and $b(n)=\lf n\a^q\rf$, where $\a$ is 
the dominant zero of $x^q-x^{q-1}-1$. This is the object of the sixth theorem. 

\smallskip

 Section 3 is an epilogue which contains precise open questions, general research projects 
and a historical note. In particular, the question of the existence of analogous 
proofs and theorems for the tribonacci and the Padovan families is raised.  
In fact, we also prove a few theorems in that direction to better explain what we mean. 
 
 The paper is elementary. Concepts are defined along the way.

\section{Six theorems concerning the Narayana family} 

\noindent{\bf Definition.} If $U=(u_k)_{k\ge0}$ is an increasing sequence of integers with $u_0=1$, then 
the {\it greedy algorithm} produces a unique representation of every positive integer $n$ 
in the form $\sum_{i\ge0}d_i u_i$, where the digits $d_i$ are nonnegative. If $u_i\le n<u_{i+1}$, 
then the algorithm retains $u_i$, that is, it takes the largest term not exceeding $n$ available and 
reiterates the process with the difference $n-u_i$. Thus, the next term selected is the unique 
$u_j$ such that $u_j\le n-u_i<u_{j+1}$. Say $u_1=3$, $u_2=7$ and $u_3>12$. Then 
$12=1\cdot u_2+1\cdot u_1+2\cdot u_0$. We say $u_i$ is a {\it summand}, or a $U$-{\it summand}, of $n$ whenever 
$d_i>0$. 

 Put $a_k:=G_{2q-2+k}$, for all $k\ge0$, so $A=(a_k)_{k\ge0}$ 
starts with the last $1$ appearing in the $G$ sequence. 
Thus, because of the $q$ consecutive terms of $G$ equal to $1$, we see that $a_k=1+k$, for $0\le k\le q$. 

\medskip

\noindent{\bf Lemma 1.} {\it Let $q\ge1$ be an integer. Then the number of binary strings 
of length $k$, where at least $q-1$ zeros separate any two ones, is $a_k$.}
\begin{proof} Put $x_k$ for the number of such strings of length $k$. If $1\le k\le q$, 
then we have the string of $k$ consecutive $0$'s and $k$ 
strings with exactly one $1$.  
Thus, $x_k=1+k=a_k$. If $k>q$, then there are $x_{k-1}$ strings ending with a $0$ and $x_{k-q}$ 
ending with a $1$, because in the latter case the $q-1$ penultimate digits must be $0$'s. Thus, 
$x_k=x_{k-1}+x_{k-q}$. Hence, $x_k=a_k$ for all $k\ge0$.    
\end{proof}

\smallskip

 We are about to show that every positive integer is uniquely representable as a sum of distinct 
$A$ numbers, where two consecutive  $A$ numbers are at least $q$ indices apart from 
each other. For instance, if $q=3$, $47=41+6+2=N_{13}+N_8+N_5$ and both differences $13-8$ and $8-5$ are 
at least $3$. This extends the classical binary and Zeckendorf representations of integers. 
We make a definition which will be handy throughout this note. 

\medskip

\noindent{\bf Definition and Remark 1.}  
A {\it $q$-representation} on $a_0,\dots,a_k$ of an integer $n$ is a sum  
$\sum_{i=0}^k\e_ia_i$ equal to $n$ where each $\e_i$ is $0$ or $1$ and 
$\e_i+\e_{i+1}+\dots+\e_{i+q-1}\le1$, for all $i\ge0$. We do not require that $\e_k$ be $1$.  
Note that $q$-representations on $a_0,\dots,a_{k-1}$ are in one-to-one correspondence 
with binary strings of length $k$ via $\sum_{i=0}^{k-1}\e_ia_i\mapsto(\e_0,\dots,\e_{k-1})$. 

\medskip
 
\begin{theorem}\label{thm:Zec+} Every integer $n\ge1$ has a unique representation, obtained 
by the greedy algorithm, of the form 
$$
\sum_{i\ge0}\e_ia_i,\quad \e_i\in\{0,1\};\quad\e_i+\e_{i+1}+\cdots+\e_{i+q-1}\le1,\;(i\ge0), 
$$ where $a_i=G_{2q-2+i}$. 
\end{theorem}
\begin{proof} Let $a_k$ be the largest summand in the greedy algorithm representation 
of $n$. Suppose $n>a_k$.  
Then, $a_j$ being the next summand, we have  $a_j\le n-a_k<a_{k+1}-a_k=a_{k+1-q}$. Hence, 
$k-j\ge k-(k-q)=q$. So the greedy 
algorithm produces a $q$-representation of $n$ since indices of successive summands 
are at least $q$ places apart. 

 Our proof shows that all integers in $[0,a_k)$ admit at least one $q$-representation on 
$a_0,\dots,a_{k-1}$.  But there are $a_k$ integers in $[0,a_k)$ as many as the number of 
$q$-representations on $a_0,\dots,a_{k-1}$ by Lemma 1. This proves the uniqueness of 
the $q$-representation of all nonnegative integers.  
\end{proof}

 Besides the famous Zeckendorf representation 
of which Theorem \ref{thm:Zec+} is a generalization, it was recently discovered \cite{Al} 
that if one uses summands that are signed Fibonacci numbers then every 
positive integer is uniquely representable into distinct summands where two summands of the same 
sign are at least four indices apart and two of distinct signs at least three indices apart. 
This was named the far-difference representation. Interestingly,  
this theorem generalizes with signed $G$ summands \cite[Theorem 1.6]{DeDo}.  

\begin{theorem}\label{thm:FarDif} Let $q\ge1$. Every integer $n\ge1$ has a unique representation 
of the form 
$$
\sum_{i\ge0}\e_ia_i,\quad \e_i\in\{0,\pm1\}, 
$$ where $a_i=G_{2q-2+i}$ and two summands of the same sign are at least $2q$ indices 
apart and two summands of different sign at least $q+1$ apart. 
\end{theorem}

\bigskip

 A composition of an integer $n\ge1$ into parts from a set $S$ of positive integers  
is an ordered sum of parts from $S$ that sum up to $n$. For instance, if $S$ is the 
set of odd integers, then $5$ admits five compositions with parts in $S$ since 
$$5=3+1+1=1+3+1=1+1+3=1+1+1+1+1.$$ Denoting the number of compositions 
of $n$ with parts in $S$ by $c_n(S)$, we have by the previous example 
$c_5(1\; \text{mod }2)=5$. The next theorem is a generalization of three well-known 
and often quoted composition results concerning Fibonacci numbers (see for instance 
\cite{Slo}, A000045, or \cite[Example 4]{WeHa}). 
\begin{align}\begin{split}\label{eq:CompF}
&F_{n+1}=c_n(1\text{ or }2),\\
&F_n=c_n(\text{odd parts}),\\
&F_{n-1}=c_n(\ge2).
\end{split}\end{align} 

\begin{theorem}\label{thm:Comp+}\footnote{The second and third statements of the theorem 
have been observed when $q=3$ 
\cite[A078012]{Slo}; it also has been observed that $c_n(1\text{ or }q)=c_{n+1}\big(1\;(\text{mod }q)\big)$ 
\cite[A000045]{Slo}} Let $q\ge1$. Compositions of $n\ge1$ into parts satisfy 
\begin{eqnarray*}
c_n(1\text{ or }q)&=&G_{n+q-1},\\
c_n\big(1\;(\text{mod }q)\big)&=&G_{n+q-2},\\
c_n(\ge q)&=&G_{n-1}.
\end{eqnarray*} 
$$
\text{In particular, }\qquad\quad c_n(1 \text{ or }q)=c_n\big(1\;(\text{mod }q)\big)+c_n(\ge q).
$$
\end{theorem}

 It would be easy to verify the three identities of Theorem \ref{thm:Comp+} using \cite[Theorem 2.1]{WeHa}, 
which gives a simple characterization of the sets $S$ for which $(c_n(S))_{n\ge1}$ is a linear 
recurrence and a simple mean of finding the associated characteristic polynomial. But we proceed differently. 
\begin{proof} We recall the bijection between 
$q$-representations on $a_0,\dots,a_{k-1}$, i.e., sums $\sum_{i=0}^{k-1}\e_ia_i$, and 
binary strings $(\e_0,\dots,\e_{k-1})$ seen in Remark 1. Strings $(\e_0,\dots,\e_{k-1})$ 
are themselves in bijection with the longer strings $(\e_0,\dots,\e_{k-1},0,\dots,0)$, where we 
added $(q-1)$ $0$'s to the right of $\e_{k-1}$. But these latter strings are in one-to-one 
correspondence with compositions of $k+q-1$ into parts $1$ or $q$. Indeed, as we read a binary string 
$\e_0,\dots,\e_{k-1},0,\dots,0$, left-to-right, map any $0$ to a part $1$ and any $1$ followed 
by $(q-1)$ $0$'s to a part $q$. We obtain a composition of $k+q-1$ into parts $1$ or $q$. This latter 
map is clearly reversible. As an illustration we show the correspondence when $q=3$ and $k=4$. 

$$
\begin{matrix}
0 &  0 & 0 & 0 & | & 0 & 0\\
1 &  0 & 0 & 0 & | & 0 & 0\\
0 &  1 & 0 & 0 & | & 0 & 0\\
0 &  0 & 1 & 0 & | & 0 & 0\\
0 &  0 & 0 & 1 & | & 0 & 0\\
1 &  0 & 0 & 1 & | & 0 & 0\\
\end{matrix}\quad\longleftrightarrow\quad
\begin{matrix}
1 & + & 1 & + & 1 & + & 1 & + & 1 & +\quad  1\\
3 & + & 1 & + & 1 & + & 1\\
1 & + & 3 & + & 1 & + & 1\\
1 & + & 1 & + & 3 & + & 1\\
1 & + & 1 & + & 1 & + & 3\\
3 & + & 3\\
\end{matrix}  
$$ By Lemma 1, the number of $q$-representations on $a_0,\dots,a_{k-1}$ is $a_k$.  
Therefore $c_{k+q-1}(1\text{ or }q)=a_k$. Putting $n=k+q-1$, we find that $c_n(1\text{ or }q)=a_{n-q+1}=
G_{n+q-1}$. If $1\le n<q$, then $c_n(1\text{ or }q)=c_n(1)=1=G_{n+q-1}$.

 We now prove the second identity of Theorem \ref{thm:Comp+}. 
Splitting the $c_n:=c_n(1\pmod q)$ compositions of $n$  
according to their first part, which may be $1,1+q,\dots,1+\l q$, where $\l=\lfloor
\frac{n-1}{q}\rfloor$, we see that $c_n=c_{n-1}+c_{n-1-q}+\cdots+c_{n-1-\l q}$. 
Similarly we find that $c_{n+q}=c_{n+q-1}+c_{n-1}+\cdots+c_{n-1-\l q}$. Subtracting 
the former identity from the latter yields 
$c_{n+q}=c_{n+q-1}+c_n$. Comparing initial conditions leads to $c_n=G_{n+q-2}$. 
 
 To obtain the number of compositions of $n$ into parts $\ge q$ we establish a bijection between 
compositions of $n$ into parts $1\pmod q$ and compositions of $n+q-1$ into parts $\ge q$. This 
interesting bijective proof is an extension of the one found by Sills \cite{Sil} between compositions 
of $n$ into odd parts and compositions of $n+1$ into parts $\ge2$. 

 To a composition $P=(p_1,p_2,\dots,p_s)$ of $n$ into $s$ parts, Sills associates in a one-to-one way 
the MacMahon binary sequence, $M=M(P)$, of length $n-1$, where each $p_i$ is successively 
transformed into a 
string of $p_i-1$ zeros followed by a one, except for the last part $p_s$ transformed into $p_s-1$ 
zeros only, an extra one being superfluous. For instance, for the composition $P=(3,1,4)$ of $n=8$, 
$M(P)=001\,1\,000$. The {\it conjugate} 
$P^c$ of $P$ is defined as the composition of $n$ whose MacMahon sequence is $M'$, where $M'$ 
is obtained from $M$ replacing ones by zeros and zeros by ones. 

 Suppose $n=p_1+p_2+\dots+p_s$ where $p_i\equiv1\pmod q$ for each $i$. Put $P=(p_1,\dots,p_s)$ 
and $M=M(P)$. Then all maximal strings of $0$'s in $M$ have a length a multiple of $q$. Thus, $M'$ has 
strings of $1$'s of such lengths as well. For instance, say $q=3$ and $n=1+4+1+7+1$, then 
$M=1000110000001$ and $M'=0111\;00111\;111\;0$. This disposition in $M'$ of $0$'s and $1$'s means that 
$P^c$ is of the form $(q_1,1,1,q_2,1,1,\dots,q_r,1,1,q_{r+1})$ for some $r\ge1$. In general, 
$P^c$ is a succession for $i=1$ to $r$ of a part $q_i$  followed by $q-1$ parts equal to $1$ 
ending with an additional part $q_{r+1}$.  
Thus, 
all parts of $P^c$ of indices $2,3,\cdots,q-1,0\pmod q$ are equal to $1$. 
In the above example, $P^c=({\bf 2},1,1,{\bf 3},1,1,{\bf 1},1,1,{\bf 2})$. Finally, 
transform $P^c$ into $P'=(q'_1,q'_2,\dots,q_{r+1}')$, where $q_i'=q_i+q-1$ for each $i$, 
including $i=r+1$. We have 
reached a composition of $n+q-1$ in which all parts are $\ge q$. The reverse map is 
well-defined. For an example with $q=3$, if $P'$ is the composition $(5,3,4)$ of twelve 
into parts all $\ge3$, then $P^c=({\bf3},1,1,{\bf 1},1,1,{\bf2})$, $M'=001\;11111\;0$, 
$M=1\,1\,0000001$ and $P=(1,1,7,1)$, a composition of $10=12-(q-1)$ into parts all $1\pmod 3$. 
Hence, for all $n\ge1$, $c_{n+q-1}(\ge q)=c_n\big(1\;(\text{mod }q)\big)=G_{n+q-2}$. Thus, 
$c_n(\ge q)=G_{n-1}$ since equality holds trivially for $n=1,\dots,q-1$. 
\end{proof}   

\medskip

\noindent{\bf Remark.} Theorem \ref{thm:Comp+} needs proper interpretation when $q=1$.  
Both $c_n\big(1\pmod 1\big)$ and $c_n(\ge 1)$ are simply the number of compositions 
of $n$ into positive parts, which is well-known to be $2^{n-1}$. One of the fastest way 
to see this, as noted by Sills \cite{Sil}, is that compositions of $n$ are in one-to-one 
correspondence with all binary sequences of length $n-1$ via the MacMahon bit sequence.   
But $c_n(1 \text{ or }1)$ means that we have two kinds of $1$'s, say blue and red. 
Thus, $c_n(1 \text{ or }1)=2^n$, each of the $n$ ones having two possible colorings. 

\bigskip

Let $U=(u_k)_{k\ge0}$ be an increasing sequence of integers with $u_0=1$. 
If $\sum_{k\ge0}d_k u_k$ is the representation of $n$ obtained by the greedy 
algorithm, then we define the {\it sum-of-digit} function, $s_U(n)$,  
as the finite sum $\sum_{k\ge0}d_k$. The {\it cumulative} sum-of-digit function,   
$S_U$, is defined by $S_U(n):=\sum_{j=0}^{n-1}s_U(j)$. Many authors have studied 
in detail the function $S_U(n)$ when $U$ is a linear recurrent sequence 
\cite{PeTi,GrTi,DrGa}. Typically, in those studies,   
$U$ has a monic characteristic polynomial, $\chi_U$, whose integral coefficients are 
nonnegative and satisfy various additional conditions. Most often  
hypotheses imply that $\chi_U$ possesses a dominant real zero. 
Then we generally obtain that as $n$ tends to infinity 
\begin{equation}\label{eq:asy}
S_U(n)=c_Un\log n+O(n),
\end{equation} for some positive constant $c_U$, 
where the $O(n)$ term may often be described with great precision.  
This is the case, in particular, of geometric sequences, or of the sequence $u_n=F_{n+2}$,  
where $F_n$ is the $n$th Fibonacci number \cite{Tro,Del,Coq}. Explicit expressions 
for $c_U$ have been given in terms of the coefficients and the dominant real 
zero, $\a$, of $\chi_U$ \cite{GrTi, Ba}. In several papers \cite{Coq, Pih, Ba2, Ba}, 
the first step in proving (\ref{eq:asy})  
begins by observing that $v_k:=S_U(u_k)$ is itself a recurring sequence with characteristic 
polynomial $\chi_U^2$. It follows that $v_k=bku_k+O(u_k)$ for some $b>0$. Since $k\log\a\sim\log u_k$, 
(\ref{eq:asy}) holds for $n=u_k$ with $c_U=b/\log\a$. In the third identity of the next theorem 
we will prove that $b_k:=S_A(a_k)$ is recurring with characteristic polynomial 
$g_q^2$ for the sequence $A$ of Lemma 1. But we do so using composition results of 
Theorem \ref{thm:Comp+}. The first identity of Theorem \ref{thm:Id+} generalizes the 
well-known facts that $2^{k+1}-1=1+2+\cdots+2^k$ or that $F_{k+2}-1=F_1+F_2+\cdots+F_k$. 
This identity and more of the kind can be found in the paper \cite{BiSp}.  
That the Fibonacci numbers pop out of summing diagonals in Pascal's triangle is often perceived 
as a surprise\footnote{This observation was made at least as early as 1877 \cite[p. 13]{Lu}}. 
On the other hand, every mathematician knows that summing binomial coefficients along the $n$th 
line produces $2^n$. Summing along lines of slope $q-1$ produces $G$ numbers, a fact recorded 
for instance in \cite{Slo}, entry A003269.

\begin{theorem}\label{thm:Id+} The identities listed below hold for all $n\ge0$ 
\begin{eqnarray}
G_{n+q}-1&=&\sum_{k=0}^nG_k,\\
G_{n+q-1}&=&\sum_{k=0}^{\lfloor n/q\rfloor}\binom{n-k(q-1)}k, \\ 
S_A(G_{n+q-1})&=&\sum_{k=0}^{\lfloor n/q\rfloor}k\binom{n-k(q-1)}k,
\end{eqnarray} with $A=(a_n)_{n\ge0}=(G_{n+2q-2})_{n\ge0}$.  
\end{theorem}
\begin{proof} The first two identities are easily proved by induction (using the 
well-known identity $\binom{n}{k}=\binom{n-1}k+\binom{n-1}{k-1}$ for the second). 
However, we provide a proof of (4) which connects it to Theorem \ref{thm:Comp+}.   
We saw that $G_{n+q-1}=c_n(1\text{ or }q)$, the number of compositions of $n$ into parts $1$ or $q$. 
We may count such compositions by splitting them according to their number of $q$-parts. 
A composition containing $k$ parts equal to $q$ contains $k+(n-kq)=n-k(q-1)$ parts. 
Since we may place those $k$ $q$'s in $n-k(q-1)$ `positions', there are $\binom{n-k(q-1)}k$ such 
compositions. Summing over all $k$'s yields the identity. 

  Finally we prove that both terms of identity (5) are annihilated by 
$g_q^2$. Early in the proof of Theorem \ref{thm:Comp+}, we used a bijection -- call it $\varphi$ -- 
between integers $\ell$ in $[0,a_n)$ and the set $\mathcal C(n+q-1)$ of all compositions   
of $n+q-1$ into parts $1$  or $q$ to establish $G_{n+q-1}=c_n(1\text{ or }q)$.  
The bijection $\varphi$ was such that if $\varphi(\ell)=c$, then $\sharp q$'s in $c$ $=s(\ell)$. That is, 
the number of parts equal to $q$ in $c$ is equal to the number of summands in the $q$-representation 
of $\ell$ on $a_0,\dots,a_{n-1}$. 
Therefore, 
\begin{equation}\label{eq:x}
S_A(a_n)=\sum_{c\in\mathcal C(n+q-1)}\sharp q'\text{s in }c.
\end{equation} Put $b_n:=S_A(a_n)$. 
Now, compositions of $n+q-1$ ending with a part $1$ account for $b_{n-1}$ of all $q$'s in 
the sum on the righthand side of (\ref{eq:x}), while  
the $a_{n-q}$ compositions ending with a $q$ account for $a_{n-q}+b_{n-q}$. Indeed, compositions of 
$n+q-1$ into parts $1$ or $q$ ending with a part $q$ are in one-to-one correspondence with compositions 
of $n-1$ into parts $1$ or $q$. But, by Theorem \ref{thm:Comp+}, there are $G_{(n-1)+q-1}=a_{n-q}$ 
such compositions. Thus, $b_n=b_{n-1}+
a_{n-q}+b_{n-q}$. That is $g_q(E)\cdot b_{n-q}=a_{n-q}$, where $E$ is the shift operator 
$E\cdot y_n=y_{n+1}$. As $(a_{n-q})$ is annihilated by $g_q$, this means that $g_q^2$ annihilates 
$(b_n)$.

 Now set $f_n(x):=\sum_{k\ge0}\binom{n-k(q-1)}kx^k$. Because 
$\binom{n-k(q-1)}{k}=\binom{n-k(q-1)-1}k+\binom{n-k(q-1)-1}{k-1}$, we see 
that
\begin{eqnarray*}
f_n(x)&=&f_{n-1}(x)+x\,\sum_{k\ge1}\binom{n-k(q-1)-1}{k-1}x^{k-1}\\
&=&f_{n-1}(x)+x\,\sum_{k\ge0}\binom{n-q-k(q-1)}{k}x^k
=f_{n-1}(x)+xf_{n-q}(x).
\end{eqnarray*}
Differentiating yields $f_n'=f_{n-1}'+xf_{n-q}'+f_{n-q}$.  
Defining $x_n$ as $f_n'(1)$, which is $\sum_{k\ge0}k\binom{n-k(q-1)}k$, we obtain 
$$
(E^q-E^{q-1}-I)\cdot x_{n-q}=f_{n-q}(1)\quad\text{ or }\quad g_q^2(E)\cdot x_n=0,
$$ since $f_n(1)=G_{n+q-1}$ is annihilated by $g_q(E)$. 
Checking that the $2q$ initial values of the two sequences in (5) coincide (they are zero for 
$-q<n<q$ and one for $n=q$) finishes the proof.
\end{proof} 

\noindent{\bf Remark.} The constant $c_A$ of (\ref{eq:asy}) has an unusually simple 
expression, namely, $c_A=(\a g'_q(\a)\log\a)^{-1}$, $\a>1$ being the dominant zero of $g_q$. 
Also, among a whole class of recurrences of order $q$, $c_A$ is minimal for $1\le q\le 3$, 
and conjecturally so for all $q$. Moreover, $c_A\sim(\log q)^{-1}$ as $q$ tends to infinity 
\cite{Ba}. 

\bigskip 

 The game of Nim in which two players take turn in removing any number of beans, or tokens, 
from one of several heaps was fully analyzed in 1902 \cite{Bou} and its beautiful analysis was 
restated in the famous book \cite{HW}. Usually it is agreed that the winner is the one who removes the 
last bean (or beans). In one of the variants, called Fibonacci Nim, 
there is a unique pile of beans. If a player removes $p$ beans then the next may remove $p'$ beans as 
long as $p'\le2p$. 
On the very first move any number of beans may be removed, but not the whole pile. 
The winning strategy, based on the Zeckendorf representation, 
is reported in several books \cite{Kn}, \cite[pp.\ 219--223]{Cor} and \cite[pp.\ 204--206]{TW}.  
Many variants of Nim, or Nim-related games, may be found in a small book by Guy, where 
the winning strategy for Fibonacci Nim and higher Fibonacci Nim is a project exercise \cite[p.\ 22]{Guy}. 
Here we observe that, at least for any $q$, $1\le q\le3$, and under the generalized rules (\ref{eq:Nim}) 
of Fibonacci Nim:   
\begin{equation}\label{eq:Nim}
p'\le\begin{cases}qp,&\text{ if }p=1,\text{ or, if }q\text{ is }1\text{ or }2;\\
qp-1,&\text{ if }p\ge2\text{ and }q=3,
\end{cases}
\end{equation} there is a common winning strategy which guarantees the first player, say A, to win 
if and only if the initial number of beans is not a $G$ number. 

 In the sequel we use the expression $G$-summand to mean an $A$-summand. 

\noindent{\bf Lemma 2.} {\it Consider the Fibonacci Nim variant 
(\ref{eq:Nim}) above with $1\le q\le3$. Suppose some player, say X, leaves a number 
of beans with least $G$-summand, $G_\ell$, and the other player, Y, removes $p$ beans with $p<G_\ell$, 
then player X can remove the least $G$-summand of $G_\ell-p$.} 
\begin{proof}  Let $G_\ell-p$ and $p$ have the $q$-representations 
\begin{eqnarray*}
G_{\ell}-p & = & G_{b_1}+\cdots+G_{b_s},\\
p & = & G_{c_1}+\cdots+G_{c_t},
\end{eqnarray*} where $b_1>\cdots>b_r\ge2q-2$ and $c_1>\cdots>c_t\ge2q-2$. Then 
$$
G_\ell=G_{b_1}+\cdots+G_{b_s}+G_{c_1}+\cdots+G_{c_t}.
$$ But, as the RHS of the above identity is not the $q$-representation of $G_\ell$, 
we must have $b_s\le c_1+q-1$. We need to verify that X is always in a position to remove $G_{b_s}$ 
beans. Indeed, 
$$
qp\ge qG_{c_1}\ge qG_{b_s-q+1}\ge G_{b_s},
$$ since when $q=2$, $qG_{b_s-q+1}\ge G_{b_s-1}+G_{b_s-2}=G_{b_s}$, and when $q=3$, 
$3G_{b_s-2}\ge2G_{b_s-2}+G_{b_s-4}=G_{b_s-1}+G_{b_s-2}\ge G_{b_s}$. The  
inequality $3G_{b_s-2}\ge2G_{b_s-2}+G_{b_s-4}$ is strict unless $b_s=6$. But the 
sixth Narayana number $N_6=3$ and, when $q=3$, removing three beans is always a legal 
move. 
\end{proof}

\begin{theorem} \label{thm:Nim} Suppose $1\le q\le3$. In the Fibonacci Nim variant 
(\ref{eq:Nim}) above, the first 
player wins if and only if the initial number of beans is not a $G$ number and if,  
every time his turn to play comes, he removes a number of beans equal to the least 
$G$-summand of the number of beans left. 
\end{theorem}
\begin{proof} Let $n=G_{a_1}+\cdots+G_{a_r}$, $a_1>\cdots>a_r\ge2q-2$, be the $q$-representation of   
the number of beans, $n$, in the heap. Assume first $r\ge2$. By hypothesis, player A removes 
$G_{a_r}$ beans. Then player B cannot remove $G_{a_{r-1}}$, the next least $G$-summand. 
Indeed, using (3) of Theorem \ref{thm:Id+}, we see that 
$G_{a_{r-1}}\ge G_{a_r+q}=1+\sum_{k=q-1}^{a_r}G_k=q+1+\sum_{k=2q-1}^{a_r}G_k$, 
which is $q+1$ in case $a_r=2q-2$ and $G_{a_r}=1$. If $G_{a_r}>1$, then 
$$
G_{a_r+q}=G_{a_r+q-1}+G_{a_r}=\begin{cases}2G_{a_r},&\text{ if }q=1;\\
2G_{a_r}+G_{a_r-1},&\text{ if }q=2;\\
2G_{a_r}+G_{a_r-1}+G_{a_r-2},&\text{ if }q=3,
\end{cases}
$$ where, for $q=3$, we used $G_{a_r+2}=G_{a_r+1}+G_{a_r-1}=G_{a_r}+G_{a_r-1}+G_{a_r-2}$. 
Hence, 
$$
G_{a_{r-1}}\ge\begin{cases}qG_{a_r}+1,&\text{ if }q=1\text{ or }2;\\
3G_{a_r},&\text{ if }q=3,
\end{cases}
$$ where we used $G_{a_r-1}+G_{a_r-2}\ge G_{a_r-1}+G_{a_r-3}=G_{a_r}$ in the case $q=3$.  

 Therefore B must remove $p$ beans with $p<G_{a_{r-1}}$. But, by Lemma 2, A can now remove the least 
$G$-summand in $G_{a_{r-1}}-p$.  

 Hence, with this strategy A is constantly in a position to play. As beans 
are diminishing at every turn, B must eventually loose. 

 Suppose now $r=1$, i.e., $n=G_{\ell}$ with $\ell\ge 2q-1$ 
so that $n>1$. By the rules of the game, A cannot remove all $n$ beans. 
Moreover, by Lemma 2, B will be in a position to play the winning strategy, i.e., 
to remove the least $G$-summand left when his turn to play comes.   
\end{proof} 

\noindent{\bf Remark.} The proof of Theorem \ref{thm:Nim} shows the stability of the strategy 
under a small variation of the rules in (\ref{eq:Nim}), namely  
\begin{equation}\label{eq:mNim} 
p'\le\begin{cases}qp-(q-1),&\text{ if }p\ge2;\\
qp, &\text{ if }p=1.
\end{cases}
\end{equation} The modified rule (\ref{eq:mNim}) has the formal advantage of being the 
same for the three values of $q\in\{1,2,3\}$. 
However, as a game, the rules (\ref{eq:Nim}) are simpler. 

\medskip

 The Wythoff game is a two-pile-of-tokens variant of the game of Nim, where a player either 
removes as many tokens from one of the two heaps, or removes the same number from both heaps. 
You win if you can leave your opponent with a `position' $\big(a(n),b(n)\big)$, where $a(n)=\lf n\a\rf$, 
and $b(n)=\lf n\a^2\rf$ with $\a$ the Golden ratio, i.e., the dominant zero of $x^2-x-1$. 
The two sequences $a(n)$ and $b(n)$ form a complementary pair of Beatty sequences: their ranges 
partition the set of natural numbers. This property carries over to all pairs $\big(a(n)=\lf n\a\rf,
b(n)=\lf n\a^q\rf\big)$, if $\a$ is the dominant zero of $x^q-x-1$, $(q\ge2)$. 

 If $f$ is a function composed of $x$ $a$-functions and $y$ $b$-functions, then, when $q=2$, 
Kimberling \cite{Ki} had proved that for all $n\ge1$
$$
f(n)=F_{x+2y-2}\,a(n)+F_{x+2y-1}\,b(n)-e_f,
$$ where $e_f$ is a nonnegative constant that depends on $f$. 

 This result has received a generalization to all $q\ge2$ \cite[Theorem 27]{Ba1}, where 
coefficients of $a(n)$ and $b(n)$ are instead $G$ numbers. Here, we only state the theorem 
for the case $q=3$, where Narayana numbers come up \cite[Theorem 12]{Ba1}.    

\begin{theorem}\label{thm:BeattyN} Let $f=\ell_1\circ\ell_2\circ\dots\circ\ell_s$, $(s\ge1)$, 
be a composite function of $x$ $a$'s and $y$ $b$'s, $x+y=s$. Then,  
$$
f(n)=N_{x+3y-2}\,a(n)+N_{x+3y}\,b(n)-N_{x+3y-3}\,n-e_f(n),
$$ where $e_f$ is a nonnegative bounded integral function of $n$ that depends on $f$.  
\end{theorem}

\section{Epilogue: Questions, comments and a historic note}

\noindent{\bf Questions.}  We wrote on the spur a few questions of non-appraised difficulty  
as food for further thought.  

\medskip

{\bf 1.} Is there a simple explicit map, or combinatorial argument, that proves the 
identity $c_n(1 \text{ or }q)=c_n\big(1\pmod q\big)+c_n(\ge q)$? 
(Note that a composition of $n$ may be made of parts that are all both $\ge q$ and $1\pmod q$).

\medskip

{\bf 2.} Is there a simple Nim game whose winning strategy corresponds to removing the least 
$G$-summand when the initial number of beans is not a $G$ number for all $q\ge1$? 

\medskip

{\bf 3.} Note that $13=F_7=N_{10}$ is both a Fibonacci and a Narayana number. 
Is it the largest instance? Since $8$ and $144$ are the only nontrivial 
Fibonacci powers \cite{Bug},  $8$ is the largest instance of a power of two and a Fibonacci. 
Let $G'$ represent the $G$-sequence 
corresponding to $q'=q+1$. Then $G_3=G'_6=8$, if $q=1$. If $q=2$, then $G_7=G'_{10}=13$. 
If $q=3$, then $G_{11}=G'_{14}=19$. Is it generally true that $G_{4q-1}=G'_{4q+2}$\footnote{This part is 
true as $G_{3q-3+i}=q+i(i+1)/2$ for all $i$, $1\le i\le q+1$. So $G_{4q-1}=G_{4q-2}+G_{3q-1}=
(q^2+7q+8)/2$, which is $G'_{4q+2}=G'_{4q'-2}=q'+(q'+1)(q'+2)/2$.}? If so, is 
it the largest instance of an integer both a $G$ and a $G'$ number? 

\medskip 

{\bf 4.} Is it possible to produce analogs of our theorems and of their proofs for other families of 
recurrences, and, in particular, for the tribonacci and the Padovan families?

\medskip 

 For instance, let $(p_n)_{n\ge0}$ be the generalized Padovan sequence, i.e., the fundamental 
recurrence associated with $x^q-x-1$, ($q\ge2$). There is a generalization of 
(\ref{eq:CompF}) and an analog of Theorem \ref{thm:Comp+}, namely 

\begin{theorem}\label{thm:CompP} We have\footnote{At least the last two statements of Theorem \ref{thm:CompP} 
have been observed when $q=3$ \cite[A000931]{Slo}}
$$
p_n=\begin{cases} c_{n-q+1}(q-1\text{ or }q),&(n\ge q),\\
c_n\big(q-1\pmod q\big),&\\
c_{n+1}\big(1\pmod {q-1}\text{ and }\not=1).
\end{cases}$$
\end{theorem}
\begin{proof} There are $c_{n-q+1}$ and $c_{n-q}$ compositions of $n\ge q$ into parts $q-1$ or $q$ 
that end, respectively, with a part $q-1$ and with a part $q$. Hence, $c_n=c_{n-q+1}+c_{n-q}$. 
Since $c_i=0=p_{q-1+i}$ for $i=1,\dots,q-2$ 
and $c_{q-1}=p_{2q-2}=1$ and $c_q=p_{2q-1}=1$, the first statement of the theorem holds. By similar 
reasoning, if $c_n$ is the number of compositions of $n$ into parts $q-1$ mod $q$, we find that 
$c_{n+q}-c_n=(c_{n+1}+c_{n-q+1}+c_{n-2q+1}+\dots)-(c_{n-q+1}+c_{n-2q+1}+\dots)$. Thus, $c_{n+q}=c_{n+1}+c_n$. 
If the $c_n$'s are compositions of $n$ into parts $q$, $2q-1$, $3q-2,\dots$, then $c_{n+q}-c_{n+1}=(c_n+c_{n-q+1}+c_{n-2q+2}+\dots)-
(c_{n-q+1}+c_{n-2q+2}+\dots)$, implying again that $c_{n+q}=c_{n+1}+c_n$. Checking that initial conditions match completes the 
proof. 
\end{proof}
{\bf 5.} But can one find an explicit bijection between the set of  
compositions of $n$ into parts $q-1\pmod q$ and 
compositions of $n+1$ into parts at least $q$ and $1$ (mod $q-1$), as Theorem \ref{thm:Comp+} did  
between compositions of $n$ into parts $1$ mod $q$ and compositions of $n+q-1$ into parts $\ge q$? 

\medskip

{\bf 6.} As $p_{n+q-1}=p_n+p_{n-1}$, we see from Theorem \ref{thm:CompP} that $c_n(q-1\text{ or }q)=c_n(q-1\pmod q)
+c_n(1\pmod{q-1}\text{ and }\not=1)$ and raise, as in question {\bf 1.}, the question of the existence of a combinatorial 
explanation of this identity.

\medskip

\noindent{\bf Comments and further results.}  We make some additional observations and comments with 
regard to the fourth question aforementioned.

 There is a theorem corresponding to Theorem \ref{thm:Zec+} for the tribonacci family \cite[p. 112]{Fr}, \cite{Le}. 
The proof in \cite{Le} was also a counting proof, but it can be replaced by a proof entirely similar 
to the one we gave for Theorem \ref{thm:Zec+}: Say $q\ge2$ and $(T_n)$ is the fundamental recurrence 
associated with $x^q-x^{q-1}-\dots-x-1$, i.e., $(T_n)$ is the generalized tribonacci sequence. 
Define $a_k=T_{k+q}$ 
for all integers $k$. Thus, $a_0=1$, $a_1=2$ and $(a_k)$ is increasing for nonnegative $k$. A lemma analogous  
to Lemma 1 says that the number, $b_k$, of binary strings of length $k$, with no $q$ consecutive $1$'s, 
is $a_k$. 
Such strings with $k\ge q$ have to end by exactly one of $0$, $01$, $011$, $\dots$, 
$01\dots1$ ($q-1$ $1$'s) and thus 
one sees that $(b_k)$ satisfies the same recursion as $(a_k)$. These strings are in one-to-one correspondence 
with (tribonacci) $q$-representations $\sum_{i=0}^{k-1}\e_ia_i$ on $a_0,\dots,a_{k-1}$, where the $\e_i$'s are 
again $0$ or $1$. The greedy algorithm applied to an integer $n\in[0,a_k)$ produces a tribonacci 
$q$-representation on $a_0,\dots,a_{k-1}$. Indeed, if $a_i<n<a_{i+1}$, then the largest summand 
$a_j$ in $n-a_i$ satisfies $j<i$,  
or else $n\ge2a_i=a_i+(a_{i-1}+\dots+a_{i-q})=a_{i+1}+a_{i-q}\ge a_{i+1}$, as $a_{i-q}\ge0$, which 
would contradict $n<a_{i+1}$. If the greedy algorithm produces successively $a_i,a_{i-1},\dots,a_{i-(q-2)}$, 
then the next summand cannot be $a_{i-q+1}$, because it would again mean that 
$n\ge a_i+a_{i-1}+\dots+a_{i-q+1}=a_{i+1}$. Therefore, we can deduce the uniqueness of the tribonacci 
$q$-representation for all positive integers using the same reasoning as in Theorem \ref{thm:Zec+}. 
This can be pursued into at least the first part of Theorem \ref{thm:Comp+} and its proof: We had proved using a bijective 
argument that $c_n(1\text{ or }q)=G_{n+q-1}$. Here we find that $c_n(\le q)=T_{n+q-1}$ by adding a $0$ 
upfront each of the $a_n$ binary strings $\e_0,\dots,\e_{n-1}$ with no $q$ consecutive $1$'s. 
Then we bijectively 
map those augmented strings into the set of compositions of $n+1$ reading a string left-to-right and 
mapping $0$ to a part $1$, $01$ to a part $2$ and so on up to strings $01\dots1$ ($q-1$ ones) which 
are mapped to a part $q$. We deduce that $c_{n+1}(\le q)=a_n=T_{n+q}$ yielding the result. 

\smallskip

 In fact, there is, for the tribonacci family, a partial analog of Theorems \ref{thm:Comp+} and \ref{thm:CompP}, i.e.,  

\begin{theorem}\label{thm:CompT} We have
\begin{eqnarray*}
T_n&=&c_{n-q+1}(1,2,\dots,q),\;(n\ge q),\\
U_n&=&c_n\big(1,2,\dots,q-1\pmod q\big)\;(n\ge1),
\end{eqnarray*} where $(U_n)$ satisfies the generalized tribonacci recursion, i.e., $U_{n+q}=U_{n+q-1}+U_{n+q-2}+
\dots+U_n$, $U_0=0$ and $U_i=2^{i-1}$, $i=1,2,\dots,q-1$.
\end{theorem}
\begin{proof} To establish the second identity, one may proceed as in Theorem \ref{thm:CompP}. 
Let $c_n$ stands for $c_n\big(1,2,\dots,q-1\pmod q\big)$. Consider $c_{n+q}-c_n=(c_{n+q-1}+c_{n+q-2}+
\dots+c_{n+1}+c_{n-1}+\dots+c_{n-q+1}+c_{n-q-1}+\dots)-
(c_{n-1}+\dots+c_{n-q+1}+c_{n-q-1}+\dots)$. Hence, $c_{n+q}=c_{n+q-1}+c_{n+q-2}+\dots+c_{n+1}+c_n$. 
\end{proof}

 The second identity of Theorem \ref{thm:CompT} was proved for $q=3$ and conjectured 
to hold for all $q\ge2$ in \cite{Ro}. Note that for $q=2$, $U_n=F_n$. 

\smallskip

Also, Duch\^ene and Rigo \cite{DuRi} proposed a three-pile-of-tokens 
variant of Wythoff's game, where the tribonacci rather than the 
Fibonacci word arises in the winning strategy. 

\medskip

\noindent{\bf Historic note.}  Beginning in  1844 the $G$-sequence for $q=2$ was referred to as the 
{\it Lam\'e series}, because Lam\'e had used these numbers to give an upper bound on the number of steps the 
euclidean algorithm took to find the greatest common divisor of two integers, until 
the French mathematician, \'Edouard Lucas, on a voyage to Italy, found them in a copy 
of the Liber Abbaci of Leonardo de Pisa, alias Fibonacci, and referred to them, starting in 1876 and 
henceforth, as the Fibonacci series. 
In the Liber Abbaci they modelled the growth rate of a population of rabbits. Amusingly, it 
was recently discovered \cite{AJ} that the $G$-sequence for $q=3$ modelled the 
growth of a population of cows in some writings of Narayana, a 14th-century prominent 
Indian mathematician, 
who wrote in Sanskrit. In Narayana's model, it took two years, i.e., two generations, 
for a newborn cow, say 
$c_0$, to become a mature cow, $C$, which would then live and reproduce one veal every year 
forever onwards. Starting with a newborn $c_0$ we represent below and symbolically the 
herd's population for the first few years 
$$
{\rm c_0,\; c_1,\; C,\; Cc_0,\; Cc_0c_1,\; Cc_0c_1C,\; Cc_0c_1CCc_0,\; 
Cc_0c_1CCc_0Cc_0c_1,\cdots}
$$  where $c_0$, $c_1$ and $C$ from one year to the next become respectively $c_1$, $C$ 
and $Cc_0$. We see that the number of cows at one generation is equal to that of the previous 
generation plus the number of newborns. But the number of newborns is precisely the number of 
cows born at least three generations before. Therefore, $N_{n+3}=N_{n+2}+N_n$.


\begin{thebibliography}{99}

\bibitem{AJ} J.-P. Allouche and T. Johnson, \emph{Narayana's cows and delayed morphisms}, 
Cahiers du GREYC, JIM 96 {\bf 4}, (1996) 2--7, 
available at http://recherche.ircam.fr/equipes/repmus/jim96/actes/Allouche.ps.

\bibitem{Al} H. Alpert, \emph{Differences of multiple Fibonacci numbers}, Integers {\bf 9} (2009), 
A57, 745--749. 

\bibitem{Ba} C. Ballot, \emph{On the average digit sum for a numeration based on a linear 
recurrence}, Unif. Distrib. Theory {\bf 11} (2016), no. 2, 125--150.  

\bibitem{Ba1} C. Ballot, \emph{On Functions Expressible as Words on a Pair of Beatty Sequences}, 
{\it J. Integer Sequences}, {\bf 20}, (2017), Article 17.4.2. 

\bibitem{Ba2} C. Ballot, `On Zeckendorf and Base $b$ Digit Sums', {\it Fibonacci Quarterly\/}, 
{\bf 51}, no. 4, (2013), 319-325. 

\bibitem{BiSp} M. Bicknell-Johnson and C. P. Spears, \emph{Classes of identities for the 
generalized Fibonacci numbers $G_n=G_{n-1}+G_{n-c}$ from matrices with constant valued 
determinants}, Fibonacci Quart., {\bf 34}, (1996), 121--128. 

\bibitem{Bou} C. L. Bouton, \emph{Nim, a game with a complete mathematical theory}, 
Ann. of Math. {\bf (2) 3}, (1901/02), 35--39. 

\bibitem{Bug} Y. Bugeaud, M. Mignotte and S. Siksek, \emph{Classical and modular approaches 
to exponential Diophantine equations. I. Fibonacci and Lucas perfect powers.}, Ann. of Math. (2)
{\bf 163} (2006), no. 3, 969--1018.

\bibitem{Coq} J. Coquet and P. van den Bosch, \emph{A summation formula involving Fibonacci 
digits}, J. Number Theory, {\bf 22} (1986), 139-146.

\bibitem{Cor} C. Corge, `\'El\'ements d'Informatique', {\it Librairie Larousse}, (1975). 

\bibitem{Del} H. Delange, \emph{Sur la fonction sommatoire de la fonction `Somme des 
Chiffres'\,'}, Enseignement Math., {\bf 21} (1975), 31-47. 

\bibitem{DeDo} P. Demontigny, T. Do, A. Kulkarni, S. J. Miller and U. Varma, \emph{A generalization 
of Fibonacci far-difference representations and Gaussian behavior}, Fibonacci Quart., 
{\bf 52} (2014), 247--273.

\bibitem{TW} D. DeTemple and W. Webb, `Combinatorial Reasoning: An Introduction to the 
Art of Counting', {\it John Wiley and Sons, Inc., Hoboken, New Jersey, 1st edition\/}, (2014).

\bibitem{DrGa} M. Drmota and J. Gajdosik, \emph{The distribution of the sum-of-digits function}, 
J. Th\'eor. Nombres Bordeaux, {\bf 10} (1998), 17--32.

\bibitem{DuRi} E. Duch\^ene and M. Rigo, \emph{A morphic approach to combinatorial games: 
the Tribonacci case}, Theor. Inform. Appl. {\bf 42} (2008), 375--393. 

\bibitem{GrTi} P. J. Grabner and R. F. Tichy, \emph{Contributions to digit expansions with 
respect to linear recurrences}, J. Number Theory, {\bf 36} (1990), 160--169.  

\bibitem{Guy} R. K. Guy, `Fair Game. How to Play Impartial Combinatorial Games', 
{\it COMAP, Inc., 60 Lowell St., Arlington, MA 02174}, (1989). 

\bibitem{Fr} A. S. Fraenkel, \emph{systems of numeration}, Amer. Math. Monthly, {\bf 92.2} (1985), 
105--114. 

\bibitem{HW} G. H. Hardy and E. M. Wright, `An Introduction to the Theory of Numbers', 
{\it Oxford at the Clarendon Press, London, 4th edition\/}, (1960). 

\bibitem{Ki} C. Kimberling,  \emph{Complementary equations and Wythoff sequences}, J. Integer Seq., 
{\bf 11}, (2008), Article 08.3.3.

\bibitem{Kn} D. E. Knuth, `The Art of Computer Programming. Vol. 1: Fundamental Algorithms',  
{\it Addison-Wesley Publishing Co., Reading, Mass.-London-Don Mills, Ont.}, (1969). 

\bibitem{Le} T. Lengyel, \emph{A counting based proof of the generalized Zeckendorf's 
theorem}, Fibonacci Quart. {\bf 44} (2006), 274--276. 

\bibitem{Lu} \'E. Lucas, \emph{Recherches sur plusieurs ouvrages de L\'eonard de Pise et 
sur diverses questions d'arithm\'etique sup\'erieure}, Bulletino di bibliografia e di storia 
delle scienze matematiche e fisiche, {\bf 10}, (1877), 129--193, 239--293. 

\bibitem{PeTi} A. Peth\"o and R. F. Tichy, \emph{On digit expansions with respect to linear 
recurrences}, J. Number Theory, {\bf 33} (1989), 243--256. 

\bibitem{Pih} J. Pihko, `On the average number of summands in the Zeckendorf representation', 
Congr. Numer. 200, (2010), 317--324.

\bibitem{Ro} N. Robbins, \emph{On Tribonacci Numbers and 3-Regular Compositions},   
Fibonacci Quart. {\bf 52} (2014), 16--19. 

\bibitem{Sil} A. V. Sills, \emph{Compositions, partitions, and Fibonacci numbers}, 
Fibonacci Quart., {\bf 49}, (2011), 348--354. 

\bibitem{Slo} N. J. A. Sloane, `The Online Encyclopedia of Integer Sequences', 
Published electronically at http://oeis.org.

\bibitem{Tro} J. R. Trollope, \emph{An explicit expression for binary digital sums}, 
Math. Mag., {\bf 41} (1968), 21-25.

\bibitem{WeHa} W. Webb and N. Hamlin, \emph{Compositions and recurrences}, 
Fibonacci Quart., 2014 Conference Proceedings, {\bf 52}, (2014), 201--204.

\bibitem{Wyt} W. A. Wythoff, \emph{A modification of the game of Nim}, Nieuw Archief voor Wiskunde, 
{\bf 2} (1905-07), 199--202. 

\bibitem{Zec} E. Zeckendorf, \emph{Repr\'esentation des nombres naturels par une somme de 
nombres de Fibonacci ou de nombres de Lucas}, Bull. Soc. Roy. Sci. Li\`ege, \textbf{41}, (1972), 
179--182.

\end{thebibliography}
\end{document}